\documentclass[10pt]{article}
\usepackage{latexsym,color,amsmath,amsthm,amssymb,amscd,amsfonts}

\newtheorem{lemma}{Lemma}
\newtheorem{remark}[lemma]{Remark}
\newtheorem{theorem}[lemma]{Theorem}

\newtheorem{lem}[lemma]{Lemma}

\newtheorem{cor}[lemma]{Corollary}

\newtheorem{conj}[lemma]{Conjecture}

\newtheorem*{remark*}{Remark}

\def\qed{\hfill $\vcenter{\hrule height .3mm
\hbox {\vrule width .3mm height 2.1mm \kern 2mm \vrule width .3mm
height 2.1mm} \hrule height .3mm}$ \bigskip}

\def\RR{\mathbb{R}}

\def\vol{{\rm Vol}}

\def\conv{{\rm conv}}

\begin{document}
\title {A short note on Godbersen's Conjecture}
\date{}
\author{S. Artstein-Avidan\thanks{Supported by ISF}}
\maketitle
%

A convex body $K \subset \RR^n$  is a compact convex set with non-empty interior. 
For compact convex sets $K_1, \ldots,K_m \subset {\mathbb R}^n$, and non-negative real numbers $\lambda_1, \ldots,\lambda_m$, a classical result of Minkowski states that the volume of $\sum \lambda_i K_i$ is a homogeneous polynomial of degree $n$ in $\lambda_i$,
\begin{equation}\label{Eq_Vol-Def}
\vol \left(\sum_{i=1}^m \lambda_i K_i\right) = 
\sum_{i_1,\dots,i_n=1}^m \lambda_{i_1}\cdots\lambda_{i_n} V(K_{i_1},\dots,K_{i_n}). 
\end{equation}
The coefficient $V(K_{i_1},\dots,K_{i_n})$, which depends solely on $K_{i_1}, \ldots, K_{i_n}$, is called the mixed volume of $K_{i_1}, \ldots, K_{i_n}$.
The mixed volume is a non-negative,
translation invariant function, monotone with respect to set inclusion,
invariant under permutations of its arguments, 
and positively homogeneous in each argument. For $K$ and $L$ compact and convex, we denote $V(K[j], L[n-j])$ the mixed volume of $j$ compies of $K$ and $(n-j)$ copies of $L$. 
One has $V(K[n]) = \vol(K)$. By Alexandrov's inequality, $V(K[j],-K[n-j])\ge \vol(K)$, with equality if and only of $K= x_0-K$ for some   $x_0$, that is, some translation of $K$ is centrally symmetric. 
For further information on mixed volumes and their properties, see Section \textsection 5.1 of  \cite{Schneider-book}.

Recently, in the paper \cite{ArtsteinEinhornFlorentinOstrover}
we have shown that for any $\lambda \in [0,1]$ and for any convex body $K$ one has that 
\[ \lambda^j (1-\lambda)^{n-j} V(K[j], -K[n-k])\le {\vol(K)}. \]
In particular, picking $\lambda = \frac{j}{n}$, we get that 
\[V(K[j], -K[n-k])\le \frac{n^n}{j^j (n-j)^{n-j}}\vol(K)\sim \binom{n}{j} \sqrt{2\pi \frac{j(n-j)}{n}}. \]

The conjecture for the tight upper bound $\binom{n}{j}$, which is what ones get for a body which is an affine image of the simplex, 
was suggested
in 1938 by Godbersen \cite{Godbersen} (and independently by Hajnal and Makai Jr. \cite{Makai}).

\begin{conj}[Godbersen's conjecture]\label{conj:god}
For any convex body $K\subset \RR^n$ and any $1\le j\le n-1$,  
\begin{equation}\label{eq:Godbersen-conj} V(K[j], -K[n-j])\le \binom{n}{j} \vol(K),\end{equation}
with equality attained only for simplices. 
\end{conj}

We mention that Godbersen \cite{Godbersen} proved the conjecture for certain classes of convex bodies, in particular for those of constant width. We also mention that the conjecture holds for $j=1,n-1$ by the inclusion $K\subset n(-K)$ for bodies $K$ with center of mass at the origin, and inclusion which is tight for the simplex, see Schneider \cite{Schneider-simplex}. The bound from \cite{ArtsteinEinhornFlorentinOstrover} quoted above seems to be the currently smallest known upper bound for general $j$. 
 
In this short note we improve the aforementioned inequality and show 
\begin{theorem}\label{theorem-on-sum}
For any convex body $K\subset \RR^n$ and for any $\lambda \in [0,1]$ one has 
\[\sum_{j=0}^n \lambda^j (1-\lambda)^{n-j} V(K[j], -K[n-j])\le {\vol(K)}. \]
\end{theorem}


The proof of the inequality will go via the consideration of two bodies,  $C\subset \RR^{n+1}$ and  $T\subset \RR^{2n+1}$. Both were used in the paper of Rogers and Shephard \cite{RS}.

We shall show by imitating the methods of \cite{RS} that 
\begin{lemma}\label{lem:vol-of-C}
Given a convex body $K\subset \RR^n$ define $C\subset \RR\times \RR^n$ by 
\[ C = \conv (\{0\}\times(1-\lambda) K \cup \{1\}\times -\lambda K). \]
Then we have
\[ \vol(C)\le \frac{\vol(K)}{n+1}.\]
\end{lemma}

With this lemma in hand, we may prove our main claim by a simple computation 

\begin{proof}[Proof of Theorem \ref{theorem-on-sum}]
\begin{eqnarray*} \vol(C) &=& \int_{0}^{1} \vol((1-\eta)(1-\lambda)K - \eta \lambda K) d\eta \\ &=& 
 \sum_{j=0}^n \binom{n}{j}(1-\lambda)^{n-j}\lambda^j V(K[j],-K[n-j]) \int_{0}^{1}(1-\eta)^{n-j}\eta^j 
 d\eta\\
 & = & \frac{1}{n+1}\sum_{j=0}^n(1-\lambda)^{n-j}\lambda^j V(K[j],-K[n-j]).
\end{eqnarray*}
Thus, using Lemma \ref{lem:vol-of-C}, we have that 
\[ \sum_{j=0}^n(1-\lambda)^{n-j}\lambda^j V(K[j],-K[n-j])\le \vol(K).\]
\end{proof} 

Before turning to the proof of Lemma \ref{lem:vol-of-C} let us state a few consequences of Theorem \ref{theorem-on-sum}. First, integration with respect to the parameter $\lambda$ yields
\begin{cor}\label{cor-AVERGAE-UNIFORM} For any convex body $K\subset \RR^n$
\[ \frac{1}{n+1}
\sum_{j=0}^n \frac{V(K[j],-K[n-j])}{\binom{n}{j}}\le \vol(K),\] which can be rewritten as 
\[ \frac{1}{n-1}
\sum_{j=1}^{n-1} \frac{V(K[j],-K[n-j])}{\binom{n}{j}}\le \vol(K).\]
\end{cor}

So, on average the Godbersen conjecture is true. 
Of course, the fact that it holds true on average was known before, but with a different kind of average. Indeed, the Rogers-Shephard inequality for the difference body, which is 
\[ \vol(K-K)\le \binom{2n}{n} \vol(K)\] 
(see for example \cite{Schneider-book} or \cite{AGMbook}) can be rewritten as
\[ \frac{1}{\binom{2n}{n}} \sum_{j=0}^n  {\binom{n}{j}} V(K[j], -K[n-j]) \le \vol(K).\]

However, our new average, in Corollary \ref{cor-AVERGAE-UNIFORM} is a uniform one, so we know for instance that the median of the sequence $( {\binom{n}{j}}^{-1}{V(K[j], -K[n-j])})_{j=1}^{n-1}$ is less than two, so that at least for one half of the indices $j=1,2,\ldots, n-1$, the mixed volumes satisfy Godbersen's conjecture up to factor $2$. More generally,  apply Markov's inequality for the uniform measure on $\{1, \ldots, n-1\}$ to get 

\begin{cor}
	Let $K\subset \RR^n$ be a convex body with $\vol(K)=1$. For at least $k$ of the indices $j=1,2,\ldots n-1$ it holds that \[  
	V(K[j],-K[n-j]) \le \frac{n-1}{n-k} \binom{n}{j}.\]
\end{cor}

We mention that   the inequality of Theorem \ref{theorem-on-sum} can be reformulated, for $K$ with $\vol(K)=1$, say,  as  
\[\sum_{j=1}^{n-1} \lambda^{j-1} (1-\lambda)^{n-j-1} [V(K[j], -K[n-j])-\binom{n}{j}] \le 0\]
So that by taking $\lambda=0,1$ we see, once again, that 
 $V(K, -K[n-1]) = V(K[n-1],-K) \le n$.

%
%
%
%
%

A key ingredient in the proof of Lemma \ref{lem:vol-of-C} is Rogers-Shephard inequality for sections and projections from \cite{RS}, which states that 

\begin{lem}[Rogers and Shephard]\label{lem-RSsec-proj}
	Let $T\subset \RR^m$ be a convex body, let $E\subset \RR^m$ be a subspace of dimension $j$. Then
	\[ \vol(P_{E^{\perp}}T)\vol(T\cap E) \le \binom{m}{j}\vol(T), \]
	where $P_{E^{\perp}} $ denotes the projection operator onto $E^{\perp}$.
\end{lem}

We turn to the proof of Lemma \ref{lem:vol-of-C} regarding the volume of $C$.

\begin{proof}[Proof of Lemma \ref{lem:vol-of-C}] 
We borrow directly the method of \cite{RS}. Let $K_1,K_2\subset \RR^n$ be convex bodies, we shall consider $T\subset \RR^{2n+1}= \RR \times \RR^n \times \RR^n$ defined by  
\[ T = \conv (\{(0,0,y); y\in K_2\} \cup \{ (1,x,-x): x\in K_1\}). \]  
Written out in coordinates this is simply 
\begin{eqnarray*}  T &=& \{ (\theta, \theta x , -\theta x + (1-\theta)y ): x\in K_1, y\in K_2\}\\
& = & \{ (\theta, w, z): w\in \theta K_1, z+w\in (1-\theta)K_2\}. \end{eqnarray*}
The volume of $T$ is thus, by simple integration, equal to 
\[ \vol(T) = \vol(K_1) \vol(K_2) 
\int_0^1 \theta^n (1-\theta)^n d\theta = \frac{n!n!}{(2n+1)!}\vol(K_1) \vol(K_2). \] 
We now take the section of $T$ by the $n$ dimensional affine subspace 
\[ E = \{(\theta_0, x, 0): x\in \RR^n\}\] and project it onto the complement $E^\perp$. We get for the section: 
\[ T\cap E = \{ (\theta_0, x, 0): x\in \theta_0 K_1 \cap (1-\theta_0)K_2\}\] 
and so $\vol_n(T\cap E) = \vol(\theta_0 K_1 \cap (1-\theta_0) K_2 )$. As for the projection, we get 
\begin{eqnarray*} P_{E^{\perp}}T &=& \{ (\theta, 0, y): \exists x {\rm ~with~} (\theta, x, y)\in T \}\\
& = &  \{ (\theta, 0, y):  \theta K_1 \cap ((1-\theta)K_2-y) \}\\
& = &  \{ (\theta, 0, y):  y\in  (1-\theta)K_2-\theta K_1 \}.\end{eqnarray*}
Thus $\vol_n(P_{E^\perp}T)= \vol((\theta, y):  y\in  (1-\theta)K_2-\theta K_1)$ which is precisely a set of the type we considered before in $\RR^{n+1}$. In fact, putting instead of $K_1$ the set $\lambda K$ and instead of $K_2$ the set $(1-\lambda)K$ we get that $P_{E^\perp}T = C$.

Staying with our original $K_1$ and $K_2$, and using the Rogers-Shephard Lemma \ref{lem-RSsec-proj} bound for sections and projections, we see that 
\[ \vol(P_{E^{\perp}}T)\vol(T\cap E) \le \binom{2n+1}{n}\vol(T), \]
which translates,  to the following inequality   
\[ \vol (\conv (\{0\}\times K_2 \cup \{1\}\times (-K_1))) \le \frac{1}{n+1} \frac{\vol(K_1)\vol(K_2)}{\vol(\theta_0 K_1 \cap (1-\theta_0)K_2)}. \]

We mention that this exact same construction was preformed and analysed by Rogers and Shephard for the special choice $\theta_0 = 1/2$, which is optimal if $K_1 = K_2$. 

For our special choice of $K_2 = (1-\lambda)K$ and $K_1 = \lambda K$ we pick $\theta_0 = (1-\lambda)$ so that the intersection in question is simply $\lambda (1-\lambda) K$, which cancels out when we compute the volumes in the numerator. We end up with 
\[ \vol (\conv (\{0\}\times (1-\lambda)K \cup \{1\}\times (-\lambda K))) \le \frac{1}{n+1} \vol(K),  \] 
which was the statement of the lemma. 
\end{proof}

Our next assertion is connected with the following conjecture regarding the unbalanced difference body 
\[ D_\lambda K = (1-\lambda)K + \lambda(-K).\]

\begin{conj}\label{unbalancedRS}
	For any $\lambda\in (0,1)$ one has
	\[ \frac{\vol(D_\lambda K)}{\vol(K)} \le  \frac{\vol(D_\lambda \Delta )}{\vol(\Delta )} \]
	where $\Delta$ is an $n$-dimensional simplex. 
\end{conj}
Reformulating,  Conjecture \ref{unbalancedRS} asks whether the following inequality  holds
\begin{eqnarray}
\sum_{j=0}^n \binom{n}{j} \lambda^j (1-\lambda)^{n-j} V_j \le \sum_{j=0}^n \binom{n}{j}^2 \lambda^j (1-\lambda)^{n-j},  
\end{eqnarray}
where we have denoted $V_j= V(K [j], -K[n-j])/\vol(K)$. 

Clearly Conjecture \ref{unbalancedRS} follows from Godbersen's conjecture. Conjecture \ref{unbalancedRS} holds for $\lambda = 1/2$ by the Rogers-Shephard difference body inequality, it holds for $\lambda = 0,1$ as then both sides are $1$, and it holds on average  over $\lambda$ by Lemma \ref{lem:vol-of-C} (one should apply Lemma \ref{lem:vol-of-C} for the body $2K$ with $\lambda_0 = 1/2$).  
 We rewrite two of the inequalities that we know on the sequence $V_j$: 
\begin{eqnarray}
\sum_{j=0}^n \lambda^j (1-\lambda)^{n-j} V_j \le \sum_{j=0}^n \binom{n}{j}  \lambda^j (1-\lambda)^{n-j}. 
\end{eqnarray}\begin{eqnarray}
\sum_{j=0}^n \binom{n}{j}   V_j \le \sum_{j=0}^n \binom{n}{j}^2 . 
\end{eqnarray}
In all  inequalities we may disregard the $0^{th}$ and $n^{th}$ terms as they are equal on both sides. We may take advantage of the fact that the $j^{th}$ and the $(n-j)^{th}$ terms are the same in each inequality, and  sum only up to $(n/2)$ (but be careful, if $n$ is odd then each term appears twice, and if $n$ is even then the $(n/2)^{th}$ term appears only once).

\begin{theorem}
	For $n=4,5$ Conjecture \ref{unbalancedRS} holds. 
\end{theorem}

\begin{proof}
	For $n=4$ We have that $V_0 = V_4 = 1$ and $V_1 = V_3$. We thus know that 
	\[ 8V_1+6V_2 \le 32+36\] 
	and that for any $\lambda \in [0,1]$ we have 
	\[ (\lambda^3(1-\lambda)+\lambda(1-\lambda)^3)V_1 + \lambda^2 (1-\lambda)^2V_2 \le 
	4(\lambda^3(1-\lambda)+\lambda(1-\lambda)^3)  + 6\lambda^2 (1-\lambda)^2. \]
	We need to prove that 
	\[ 4(\lambda^3(1-\lambda)+\lambda(1-\lambda)^3)V_1 + 6\lambda^2 (1-\lambda)^2V_2 \le 
	16(\lambda^3(1-\lambda)+\lambda(1-\lambda)^3)  + 36\lambda^2 (1-\lambda)^2. \]
	If we find $a,b\ge 0$ such that
	 \[ (\lambda^3(1-\lambda)+\lambda(1-\lambda)^3)a + 8b = 4(\lambda^3(1-\lambda)+\lambda(1-\lambda)^3)\] and 
	 \[  \lambda^2 (1-\lambda)^2a + 6b = 6\lambda^2 (1-\lambda)^2\] then by summing the two inequalities with these coefficients, we shall get the needed inequality.
	 
	 We thus should check whether the following system of equations has a non-negative solution in $a,b$:
	 \[ \left(\begin{array}{ll}
	 (\lambda^3(1-\lambda)+\lambda(1-\lambda)^3 &	8  \\ 
	 \lambda^2 (1-\lambda)^2	 & 6 \\ 
	 \end{array} \right)\left(\begin{array}{l}
	 a  \\ 
	 b\\ 
	 \end{array}\right) = \left(\begin{array}{l}
	 4(\lambda^3(1-\lambda)+\lambda(1-\lambda)^3) \\ 
	 6\lambda^2 (1-\lambda)^2 \\ 
	 \end{array}\right) . \]
	 
	The determinant of the matrix of coefficients is positive: 
	\begin{eqnarray*} 6(\lambda^3(1-\lambda)+\lambda(1-\lambda)^3)-8\lambda^2 (1-\lambda)^2  = 
	\\
	2\lambda (1-\lambda)  [3(\lambda^2 + (1-\lambda)^2)-4\lambda(1-\lambda)] = \\
	 	2\lambda (1-\lambda)  [3(1-2\lambda)^2 +2 \lambda(1-\lambda)]  \ge 0 
	\end{eqnarray*}
	 We invert it to get, up to a positive multiple, that
 \begin{eqnarray*} \left(\begin{array}{l}
 a  \\ 
 b\\ 
 \end{array}\right) & =& c \left(\begin{array}{ll}
  6  &	-8  \\ 
 -\lambda^2 (1-\lambda)^2	 & (\lambda^3(1-\lambda)+\lambda(1-\lambda)^3 \\ 
 \end{array} \right) \left(\begin{array}{l}
 4(\lambda^3(1-\lambda)+\lambda(1-\lambda)^3) \\ 
 6\lambda^2 (1-\lambda)^2 \\ 
 \end{array}\right) \\
 & = &  c       \left(\begin{array}{l}
 	24(\lambda^3(1-\lambda)+\lambda(1-\lambda)^3) - 48\lambda^2 (1-\lambda)^2  \\ 
 	 2(\lambda^3(1-\lambda)+\lambda(1-\lambda)^3)\lambda^2 (1-\lambda)^2  \\ 
 \end{array}\right)    \\
 & = &  c       \left(\begin{array}{l}
 	24\lambda (1-\lambda) (1-2\lambda)^2   \\ 
 	2(\lambda^3(1-\lambda)+\lambda(1-\lambda)^3)\lambda^2 (1-\lambda)^2  \\ 
 \end{array}\right)  .  \end{eqnarray*}
 We see that indeed the resulting $a,b$ are non-negative.

 For $n=5$ we do the same, namely we have $V_0 = V_5 = 1$ and $V_1= V_4$ and $V_2 = V_3$ so we just have two unknowns, for which we know that 
 	\[ 5V_1+10 V_2 \le 25+100 \] 
 	and that for any $\lambda \in [0,1]$ we have 
 	\[ (\lambda^4(1-\lambda)+\lambda(1-\lambda)^4)V_1 + (\lambda^2 (1-\lambda)^3+ \lambda^3(1-\lambda)^2)V_2 \le 
 	5(\lambda^4(1-\lambda)+\lambda(1-\lambda)^4)  + 10(\lambda^2 (1-\lambda)^3+ \lambda^3(1-\lambda)^2). \]
 	We need to prove that 
 	\[ 5(\lambda^4(1-\lambda)+\lambda(1-\lambda)^4)V_1 + 10 (\lambda^2 (1-\lambda)^3 + \lambda^3 (1-\lambda)^2)V_2 \le 
 	25(\lambda^4(1-\lambda)+\lambda(1-\lambda)^4) + 100 (\lambda^2 (1-\lambda)^3 + \lambda^3 (1-\lambda)^2). \]
 We are thus looking for a non-negative solution to the equation 
 \[ \left(\begin{array}{ll}
 (\lambda^4(1-\lambda)+\lambda(1-\lambda)^4)  &	5   \\ 
 	(\lambda^2 (1-\lambda)^3 + \lambda^3 (1-\lambda)^2) &  10 \\ 
 \end{array} \right)\left(\begin{array}{l}
 a  \\ 
 b\\ 
 \end{array}\right) = \left(\begin{array}{l}
5 (\lambda^4(1-\lambda)+\lambda(1-\lambda)^4)  \\ 
10(\lambda^2 (1-\lambda)^3 + \lambda^3 (1-\lambda)^2) \\ 
 \end{array}\right) . \]
 The determinant is positive since the left hand column is decreasing and the right hand column increasing. Up to a positive constant $c$ we thus have 
 
  \[ \left(\begin{array}{l}
  a  \\ 
  b\\ 
  \end{array}\right) = c \left(\begin{array}{ll}
  10   &	-5   \\ 
 - (\lambda^2 (1-\lambda)^3 + \lambda^3 (1-\lambda)^2) & (\lambda^4(1-\lambda)+\lambda(1-\lambda)^4) \\ 
  \end{array} \right) \left(\begin{array}{l}
  5 (\lambda^4(1-\lambda)+\lambda(1-\lambda)^4)  \\ 
  10(\lambda^2 (1-\lambda)^3 + \lambda^3 (1-\lambda)^2) \\ 
  \end{array}\right) . \]
 Multiplying we see that the solution is non-negative. (We use that $(\lambda^j (1-\lambda)^{n-j} + \lambda^{n-j} (1-\lambda)^j)$ is decreasing in $j\in \{0,1,\ldots, n/2\}$, an easy fact to check.)
\end{proof}

	We end this note with a simple geometric proof of the following inequality from \cite{ArtsteinEinhornFlorentinOstrover} (which reappeared independently in \cite{AGJV}) 
	\begin{theorem} \label{thm:strange} Let  $K,L \subset {\mathbb R}^n$ be convex bodies which include the origin. Then 
		$$
		\vol (\conv(K \cup - L) ) \,
		\vol ((K^\circ +  L^\circ)^\circ ) \le
		\vol(K) \, \vol(L).$$
	\end{theorem}
	We remark that this inequality can be thought of as a dual to the Milman-Pajor inequality \cite{MilmanPajor} stating that when $K$ and $L$ have center of mass at the origin one has 
	\[ \vol (\conv(K \cap - L) ) \,
			\vol (K +  L) \ge
			\vol(K) \, \vol(L).\]

	\begin{proof}[Simple geometric proof of Theorem \ref{thm:strange}]
	Consider two convex bodies $K$ and $L$ in $\RR^n$ and build the body iin $\RR^{2n}$ which is 
	\[ C = \conv (K\times \{0\}\cup \{0\}\times L)\]
	The volume of $C$ is  simply 
	\[ \vol(C) = \vol(K) \vol(L) \frac{1}{\binom{2n}{n}}. \]
	Let us look at the two orthogonal subspaces of $\RR^{2n}$ of dimension $n$ given by $E = \{ (x,x): x\in \RR^n\}$ and $E^{\perp} = \{ (y,-y): y \in \RR^{n}\}$. First we compute $C\cap E$: 
	
	\[ C\cap E = \{ (x,x): x = \lambda y, x=(1-\lambda)z,\lambda \in [0,1], y\in K, z\in L\}. \] In other words, 
	\[ C\cap E = \{ (x,x): x\in \cup_{\lambda\in [0,1]}( \lambda K \cap (1-\lambda)L)\}= 
	 \{ (x,x): x\in (K^{\circ}+L^{\circ})^{\circ}\}.
	\] 
	Next let us calculate the projection of $C$ onto $E^{\perp}$: 
	Since $C$ is a convex hull, we may project $K\times \{0\}$ and $\{0\}\times L$ onto $E^{\perp}$ and then take a convex hull. In other words we are searching for all $(x,-x)$ such that there exists $(y,y)$ with $(x+y, -x+y)$ in  $K\times \{0\}$ or 
	$\{0\}\times L$. Clearly this means that $y$ is either $x$, in the first case, or $-x$, in the second, which means we get 
	\[P_{E^{\perp}}C = \conv \{ (x,-x): 2x\in K {\rm ~or~}
	-2x \in L\} = \{ (x,-x): x\in \conv (K/2 \cup -L/2) \}.\]

	In terms of volume we get that 
	\[ \vol_n(C\cap E) = \sqrt{2}^n \vol_n((K^{\circ}+L^{\circ})^{\circ}) \] and 
	\[ \vol_n (P_{E^{\perp}}(C)) = \sqrt{2}^{-n} \vol_n (\conv (K \cup -L))\]
	and so their product is precisely the quantity in the right hand side of Theorem \ref{thm:strange}, and by the Rogers Shephard inequality for sections and projections, Lemma \ref{lem-RSsec-proj}, we know that
	\[ \vol_n(C\cap E)\vol_n (P_{E^{\perp}}(C))
	\le \vol_n(C)\binom{2n}{n}. \]
	Plugging in the volume of $C$, we get our inequality from Theorem \ref{thm:strange}.\end{proof}
	
	\begin{remark}
	Note that taking, for example, $K = L$ in the last construction, but taking $E_\lambda = \{(\lambda x, (1-\lambda)x): x\in \RR^n\}$, we  get that 
	\[ C\cap E = \{ (\lambda x, (1-\lambda)x): x\in K\}\] and 
	\[ P_{E_\lambda^\perp} C = \{ ((1-\lambda)x, -\lambda x): x \in \frac{1}{\lambda^2 + (1-\lambda)^2} \conv((1-\lambda)K \cup -\lambda K)
	\}. \]
	In particular, the product of their volumes, which is simply 
	\[ \vol( \conv((1-\lambda)K \cup -\lambda K))\vol(K)\]
	is bounded by $\binom{2n}{n}\vol(C)$ which is itself $\vol(K)$, giving yet another proof of the following inequality from \cite{ArtsteinEinhornFlorentinOstrover}, valid for a convex body $K$ such taht $0\in K$
	   \[\conv((1-\lambda)K \cup -\lambda K) \le \vol(K),
	\]
	 and more importantly a realization of all these sets as projections of a certain body. 
	\end{remark}

{\small
\noindent Shiri Artstein-Avidan 
\noindent School of Mathematical Science, Tel Aviv University, Ramat
Aviv, Tel Aviv, 69978, Israel.\vskip 2pt
\noindent Email address: shiri@post.tau.ac.il}
 
\end{document}